\documentclass[12pt]{amsart}
\usepackage[
backend=biber,
style=alphabetic,
sorting=ynt
]{biblatex}
\usepackage[english]{babel} 
\usepackage[utf8]{inputenc}

\usepackage{enumitem}
\usepackage{amsfonts}
\usepackage{amsmath}
\usepackage{amssymb}
\usepackage{amsthm}
\usepackage{mathtools}
\usepackage{stmaryrd}
\usepackage{hyperref}
\hypersetup{
    colorlinks=false,
    linkcolor=blue
}
\usepackage[capitalize]{cleveref}
\usepackage{tikz}
\input tikz.tex
\usetikzlibrary{cd}

\usepackage{graphicx}
\usepackage[letterpaper, margin=1in, total={6in, 8in}]{geometry}

\renewcommand{\phi}{\varphi}

\DeclarePairedDelimiter\floor{\lfloor}{\rfloor}
\DeclarePairedDelimiter\ceil{\lceil}{\rceil}

\usepackage[OT2,T1]{fontenc}

\setlist{noitemsep}

\numberwithin{figure}{section}
\numberwithin{equation}{section}

\theoremstyle{definition}
\newtheorem{theorem}[figure]{Theorem}
\newtheorem{lemma}[figure]{Lemma}

\newtheorem{corollary}[figure]{Corollary}
\newtheorem{proposition}[figure]{Proposition}
\newtheorem{definition}[figure]{Definition}

\newtheorem{remark}[figure]{Remark}
\newtheorem{example}[figure]{Example}

\title{Skelet \#17 and the fifth Busy Beaver number}
\author{Chris Xu}
\date{\today}
\email{chx007@ucsd.edu}

\begin{document}
\maketitle

\section{Description}

\subsection{Introduction}
The \textit{busy beaver} function $\text{BB}(n)$ is defined to be the maximum number of moves a halting $n$-state $2$-symbol Turing machine makes before it halts. Although there is no general algorithm to compute $\text{BB}(n)$ in general, researchers have long conjectured that $\text{BB}(5) = \text{47,176,870}$, perhaps most recently by Aaronson \cite{bbfrontier}.

A massive undertaking by The Busy Beaver Challenge team \cite{bbchallenge} has narrowed down the determination of $\text{BB}(5)$ to showing non-halting of approximately 30 ``holdout'' machines \cite{bbchallengeholdouts}. Of those machines, there is significant overlap with a different list of 43 holdouts for which Georgi ``Skelet'' Georgiev claimed to find in 2003 \cite{skelethomepage}: the machine dubbed Skelet \#17 \cite{skelet17} is the focus of our paper. 

Although most of the 30 holdouts can be grouped into well-known categories, analysis of Skelet \#17 has eluded all known attempts at classification. Members of \cite{bbchallenge} are in general agreement that Skelet \#17, along with Skelet \#1, comprise the two most difficult holdouts to analyze. In this paper, we show that Skelet \#17 does not halt (\cref{thm:mainthm}). 

\subsection{Skelet \#17}
By \cite{savask}, is known that Skelet \#17 can be described by the following process: begin with the state $S = (0,2,4,0)$. Let $P(S)$ be the next state after $S$. $P(S)$ is defined by the following rules:
\begin{itemize}
    \item \underline{Overflow}: If $S = (2a_\ell+1, 2a_{\ell-1}, \dots, 2a_0)$, transition to $P(S) = (0,2a_\ell+2, 2a_{\ell-1}, \dots, 2a_0)$.
    \item \underline{Halt}: If $S = (0,0,2a_{\ell-2}, \dots, 2a_0)$, halt.
    \item \underline{Zero}: If $S = (2a_\ell, 2a_{\ell-1}, \dots, 2a_0)$ and $(a_1,a_2) \neq (0,0)$, transition to $P(S) = (0,0,2a_\ell+1, 2a_{\ell-1}, \dots, 2a_1, 2a_0-1)$.
    \item \underline{Halve}: If $S = (a_\ell, \dots, a_1, -1)$, transition to $P(S) = (a_\ell, \dots, a_1)$.
    \item \underline{Increment}: If $S = (a_\ell,\dots,a_1,a_0)$ is not in the form specified by any of the above rules, find the rightmost index $a_i$ of $S$ with an odd value. Transition to $P(S) = (a_\ell,a_{i+2},a_{i+1}+1,a_i,\dots,a_1,a_0-1).$
\end{itemize}
\begin{remark}
    Note that the conventions of \cite{savask} are slightly different from ours: the precise correspondence is as follows:
    \begin{itemize}
        \item (E1) and (E2) in \cite{savask} correspond to \underline{Zero} here.
        \item (E3) in \cite{savask} corresponds to \underline{Halt} here.
        \item (O1) in \cite{savask} corresponds to \underline{Zero} $\circ$ \underline{Overflow} here.
        \item (O2), (O3), (O4) and (O5) in \cite{savask} correspond to \underline{Increment} here.
        \item (O6) in \cite{savask} corresponds to \underline{Halve} $\circ$ \underline{Increment} here.
    \end{itemize}
\end{remark}
\begin{definition}
    Given two states $S$ and $T$, say that $S \mapsto T$ if $P^k(S)=T$ for some $k \in \mathbb{N}$. In this case, define $S \to T$ to be the sequence of transition rules that were applied from $S$ to obtain $T$.
\end{definition}

The goal of this paper is to show that the above process never terminates. In particular, we will establish the following theorem:
\begin{theorem} \label{thm:mainthm}
    We have $(0,2,4,\dots,2^{2k},0) \mapsto (0,2,4,\dots,2^{2k+2},0)$ for all $k \in \mathbb{N}$. Moreover, during this, we use exactly one \underline{Overflow} rule. As a result, Skelet \#17 never halts.
\end{theorem}

\subsection{Formalization}
The author originally wrote up a sketch of this paper in 2023. Recently, ``Mxdys'' has adapted its contents to a formal proof of Skelet \#17's nonhalting \cite{formalproof17}, along with formalizing everything else in the $\text{BB}(5)$ program \cite{formalproof}. As a result, we are able to conclude:
\begin{theorem}
    We have $\text{BB}(5) = 47,176,870$.
\end{theorem}

\section{Basics}

\subsection{State variables}
For $r \in \mathbb{R}$, let $\left<r\right>$ denote the nearest integer to $r$, where we round up if $r$ is a half-integer (e.g $\left<1.5\right> = 2$). 
\begin{definition} \label{def:graycode}
    For $n \in \mathbb{N}$, the \textit{Gray code} of $n$ is defined to be $\text{GrayCode}(n) := \cdots a_3 a_2 a_1 a_0$, where each digit $a_i \in \{0,1\}$ equals the mod $2$ reduction of $\left<n/2^i\right>$.
\end{definition}
For a state $S = (a_\ell,\dots, a_0)$, introduce the following notation, which is specified in order to capture the most important properties of Skelet \#17:
\begin{itemize}
    \item Let $n := n(S)$ denote the number $\text{GrayCode}^{-1}(\overline{a_\ell}\overline{a_{\ell-1}}\dots \overline{a_1})$. Here we let $\overline{a_i} := a_i \bmod 2$, and note that we are not considering $a_0$ in the formation of $n$.
    \item Let $\ell := \ell(S)$ denote the number one less than the size of $S$, i.e. $|S|-1$.
    \item Let $\sigma := \sigma(S) \in \{-1, +1\}$ be $+1$ if $\sum_{i=0}^\ell a_i$ is odd, and $-1$ if $\sum_{i=0}^\ell a_i$ is even.
    \item In addition to the three state variables above, let $a_i := a_i(S)$ denote the value $a_i$ of the $i^{th}$ index of $S$, for $i \in [0,\ell].$
\end{itemize}
Under the transition rules, we may write down how the first three state variables change for each rule that is applied. In particular, one can verify the following table (where the substitutions in a given rule's row are performed going left to right):
\begin{center}
\begin{tabular}{ c|c|c|c}
    Rule & $n$ & $\ell$ & $\sigma$ \\
    \hline 
    \underline{Overflow} & $2^\ell-1 \mapsto 0$ & $\ell \mapsto \ell+1$ & $+1 \mapsto -1$\\
    \underline{Empty} & $0 \mapsto 2^\ell-1$ & $\ell \mapsto \ell+2$ & $-1 \to -1$ \\
    \underline{Halve} & $n \mapsto \floor{n/2}$ & $\ell \mapsto \ell-1$ & $\sigma \mapsto -\sigma$ \\
    \underline{Increment} & $n \mapsto n + \sigma$ & $\ell \mapsto \ell$ & $\sigma \mapsto \sigma$ \\
    \underline{Halt} & $0 \mapsto \text{N/A}$ & $\ell \mapsto \text{N/A}$ & $-1 \mapsto \text{N/A}$
\end{tabular}
\end{center}
The above table does more than just specify how the state variables change with each rule: it also restricts what $(n,\ell,\sigma)$ can be immediately before a given rule is applied. For instance, if $S$ is a state for which we \underline{Overflow} out of, then we must necessarily have $n(S) = 2^{\ell(S)}-1$ and $\sigma(S) = +1$.

\subsection{Increments}
Define the function
\begin{align*}
    d_j(a,b) &:= \left|\left<\frac{a}{2^j}\right> - \left<\frac{b}{2^j}\right>\right|.
\end{align*}
The next proposition states how a state's coordinates change under a series of increments:
\begin{proposition} \label{prop:incr}
    Suppose we have states $S$ and $S'$ where $S \mapsto S'$, such that $S \to S'$ consists of rules of the form Increment. Denote $n := n(S)$, $n' := n(S')$ and $\sigma := \sigma(S) = \sigma(S')$. Then for each $i \in [0,\ell]$ we have 
    \begin{align*}
        a_i(S') &= \begin{cases}
            a_i(S) + d_i(n,n') & i \geq 1 \\
            a_i(S) - d_i(n,n') & i = 0.
        \end{cases}
    \end{align*}
\end{proposition}
\begin{proof}
    For $i = 0$, we have $d_i(n,n') = |n'-n|$, which is precisely the number of \underline{Increment} rules in $S \to S'$. Since $a_0$ decreases after an \underline{Increment} rule, the claim $d_0(n,n') = a_0(S) - a_0(S')$ follows immediately. For $i > 0$, note that by \cref{def:graycode}, the value $d_i(n,n')$ quantifies how much $a_i$ increases when incrementing the Gray code of $n$ to the Gray code of $n'$.
\end{proof}

\section{The speedup}
\subsection{Conventions}
From this section onwards, fix $k \in \mathbb{N}$ and let $S_k := (0,2,2^2,\dots,2^{2k},0)$. From now on, every state $E$ we will consider will satisfy $S_k \mapsto E$, but not $S_{k+1} \mapsto E$. For $E = (a_\ell, \dots, a_1, a_0)$ and $i \in [0,\ell]$, denote $E[i] := (a_\ell, \dots, a_{i+1}, a_i+2, a_{i-1}, \dots, a_0)$, which is $E$ but with $a_i$ incremented by $+2$.

\subsection{Empty and embanked states}
\begin{definition}
    Call a state $E$ \textit{empty} if $n(E) = 0$ and $\sigma(E) = -1$, but the next rule applied is not \underline{Halt} (so in particular the state immediately after $E$ is achieved by applying \underline{Zero}).
\end{definition}
If $E$ is empty, then let $N(E) := T_E(E)$ denote the next state after $E$ that is empty. Let $T_E$ denote the sequence of rules $E \to N(E)$ if $N(E)$ exists, and otherwise let $T_E$ denote the sequence of rules after $E$.

\begin{definition} 
    Let $E$ be an empty state.
    \begin{enumerate}
        \item Say that $E$ is \textit{embanked} if the non-\underline{Increment} rules of $T_E$ consist of exactly one \underline{Zero} rule at the start and exactly two \underline{Halve} rules elsewhere.
        \item Say that $E$ is \textit{weakly embanked} if the rules of $T_E$ consist of one \underline{Zero} rule at the start and at least two \underline{Halve} rules (in particular, there may be arbitrarily many), such that all other rules before the second \underline{Halve} rule are \underline{Increment} rules.
        \item Suppose $E$ is weakly embanked. For $i \in \{1,2\}$, define $h_i := h_i(E)$ (resp. $s_i := s_i(E)$) to be the value of $n$ for the state immediately after (resp. before) the $i^{th}$ \underline{Halve} rule is applied in $T_E$. Let $h := h(E)$ (resp. $s := s(E)$) denote the tuple $(h_1(E), h_2(E))$ (resp. $(s_1(E), s_2(E))$). So in particular we have $h_i = \lfloor s_i/2 \rfloor$.
    \end{enumerate}



    
\end{definition}


In the course of proving \cref{thm:mainthm}, we will in fact see that the vast majority of the empty states we consider are embanked. In the first non-embanked empty state we see, an \underline{Overflow} rule results, at which point an explicit analysis will yield the desired result. It is in this fact that the reason for the term \textit{embanked} becomes clear: it signifies that $T_E$ flows extremely well as we repeatedly apply $N(\cdot)$, and in particular well enough for us to apply major speedups to the Skelet \#17 process.
\begin{example}
    Let $E := (2,2,6,8,18,0)$. Then $E$ is embanked with $h(E) = (15, 17)$. Let us spell out $T_E$ in detail: $n$ will start at $31$ after the first \underline{Empty} rule, with corresponding state $(0,0,3,2,6,8,18,-1)$. Here, we have $a_0 = -1$, so we must immediately apply the first \underline{Halve} rule, which takes $n$ to $\lfloor 31/2 \rfloor = 15$ and $\sigma$ to $+1$, with ensuing state $(0,0,3,2,6,8,18)$. A series of \underline{Increment} rules are applied until $n=34$, where the state is now $(1,1,4,4,11,17,-1)$. Then the second \underline{Halve} rule is applied to yield $(n,\sigma) = (17,-1)$, and after a further series of \underline{Increment} rules we have $n = 0$, which corresponds to the state $E' = (2,2,6,8,20,0) = E[1]$. The reader is highly encouraged to graph the values of $n$ attained in $T_E$ to gain an intuition behind this paper.
\end{example}

\begin{proposition} \label{prop:weaklyembankedbounds}
    Empty state $E$ is weakly embanked if and only if the following conditions hold:
    \begin{enumerate}
        \item $a_0(E) < 2^{2k+1}-1$.
        \item $a_1(E) < 3 \cdot 2^{2k} - 1$.
    \end{enumerate}
\end{proposition}
\begin{proof}
    For $E$ to be weakly embanked, the sequence $T_{\underline{\text{Zero}}(E)}$ must have two \underline{Halve} rules before any non-\underline{Increment} rule. To guarantee that the first \underline{Halve} rule occurs before any non-\underline{Increment} rule, the variable $a_0$ must decrement to $-1$ from $a_0(\underline{\text{Zero}}(E)) = a_0(E) - 1$ before the variable $n$ decrements to $0$ from $n(\underline{\text{Zero}}(E)) = 2^{2k+1}-1$. This is the same as the first condition.

    Assume condition (1) holds, and let $E'$ be the state immediately after applying the first \underline{Halve} rule in $T_E$. To guarantee that the second \underline{Halve} rule occurs before any non-\underline{Increment} rule, the variable $a_0$ must decrement from $a_0(E')$ to $-1$ before the variable $n$ increments from $n(E')$ to $2^{2k+2}-1$ (after which an \underline{Overflow} rule would have to be applied). We readily compute
    \begin{align*}
        a_0(E') &= a_1(E) + d_1(2^{2k+1}-1, 2^{2k+1}-1-a_0) \\
        n(E') &= \left\lfloor \frac{2^{2k+1}-1-a_0(E)}{2} \right\rfloor
    \end{align*}
    so that the desired condition becomes the inequality 
    \begin{align*}
        a_1(E) + d_1(2^{2k+1}-1, 2^{2k+1}-1-a_0(E)) + 1 < 2^{2k+2}-1 - \left\lfloor \frac{2^{2k+1}-1-a_0(E)}{2} \right\rfloor.
    \end{align*}
    Note that $a_0(E)$ is necessarily even (since $E$ is empty), so the above inequality becomes
    \begin{align*}
        a_1(E) + \frac{a_0(E)}{2} + 1 &< 2^{2k+2}-1-\left(2^{2k} - \frac{a_0(E)+2}{2}\right),
    \end{align*}
    which further simplifies to $a_1(E) < 3\cdot 2^{2k} - 1.$ This is exactly condition (2), so we win.
\end{proof}

\subsection{Speedup}
\begin{lemma} \label{lemma:change}
    Let $E = (a_\ell, \dots, a_1, a_0)$ be an embanked state with $h(E) =: (h_1, h_2)$, $s(E) =: (s_1, s_2)$, and choose $i \in [0,\ell]$ such that $E[i]$ is weakly embanked. Then we have
    \begin{align*}
        h(E[i]) &= \begin{cases}
            (h_1-1, h_2) & i = 0 \\ 
            (h_1, h_2+1) & i = 1 \\
            (h_1, h_2) & i \geq 2.
        \end{cases} \\
        s(E[i]) &= \begin{cases}
            (s_1-2, s_2) & i = 0 \\ 
            (s_1, s_2+2) & i = 1 \\
            (s_1, s_2) & i \geq 2.
        \end{cases}
    \end{align*}
\end{lemma}
\begin{proof}
    Note that $h(E)$ depends on only the indices $a_1$ and $a_0$ of $E$ (this immediately implies $h(E[i]) = h(E)$ for $i \geq 2$). In particular, $a_0$ influences $h_1$ while $a_1$ influences $h_2$. Thus, if $i = 0$, then, before the first \underline{Halve} rule of $T_{E[i]}$, we must apply the \underline{Increment} rule two additional times compared to $T_E$ before the necessary $a_0 = -1$ occurs; it follows that $h(E[0]) = (h_1-1,h_2)$. A similar analysis for the states preceding the second \underline{Halve} rule of $T_{E[i]}$ yields $h(E[1]) = (h_1, h_2+1).$
\end{proof}


\begin{proposition} \label{prop:nextindex}
    Suppose $E$ is an embanked state such that $N(E) = E[i]$ for some $i \in [0,\ell(E)]$, and suppose $N(E)$ is weakly embanked. Let $\ell = \ell(E)$, $h(E) = (h_1, h_2)$ and $s(E) = (s_1, s_2)$. Then $N(E)$ is in fact embanked, with
    \begin{align*}
        N(N(E)) &= \begin{cases}
            N(E)[\nu_2(h_1)] & i = 0 \\
            N(E)[\nu_2(h_2+1)+1] & i = 1 \\
            N(E)[i- 2] & i \geq 2.
        \end{cases}
    \end{align*}
\end{proposition}
\begin{proof}
    If $i\geq 2$, then by \cref{lemma:change}, $h(N(E)) = h(E)$. $N(E)$ is just $E$ with $a_i$ incremented by $+2$, and the $i^{th}$ index will shift to the $(i-2)^{th}$ index upon another application of $N$. Since the indices $a_1$ and $a_0$ are the same for $E$ and $N(E)$, we have $T_E = T_{N(E)}$, hence $N(E)$ is embanked with $N(N(E)) = N(E)[i-2]$. 
    
    The nontrivial cases are $i \in \{0,1\}$. If $i = 1$, then \cref{lemma:change} implies that, compared to $T_E$, the transition $T_{E[1]}$ acquires exactly two additional \underline{Increment} rules between the first and second \underline{Halve} rules, as well as one additional \underline{Increment} rule after the second \underline{Halve} rule. Let us analyze what happens to the $j^{th}$ index for each $j \in [0,\ell]$. Let $T_E'$ denote the transition $\underline{\text{Zero}}(E) \to N(E)$; this is $T_E$ but with the first transition removed. By \cref{prop:incr}, we have 
    \begin{align*}
        a_j(E[1]) &= \begin{cases}
            a_{j+2}(\underline{\text{Zero}}(E)) + d_{j+2}(2^{\ell}-1,s_1) + d_{j+1}(h_1, s_2) + d_j(h_2,0) & j \geq 1 \\
            a_{j+2}(\underline{\text{Zero}}(E)) + d_{j+2}(2^{\ell}-1,s_1) + d_{j+1}(h_1, s_2) - d_j(h_2,0) & j = 0,
        \end{cases}
    \end{align*}
    noting that (1) the indices shift every time a \underline{Halve} rule is applied, and (2) all $d_j$'s are in fact defined, since \underline{Zero} adds two extra indices to the state. Moreover, define
    \begin{align*}
        a_j'(N(E[1])) &:= \begin{cases}
            a_{j+2}(\underline{\text{Zero}}(E[1])) + d_{j+2}(2^{\ell}-1,s_1) + d_{j+1}(h_1, s_2+2) + d_j(h_2+1,0) & j \geq 1 \\
            a_{j+2}(\underline{\text{Zero}}(E[1])) + d_{j+2}(2^{\ell}-1,s_1) + d_{j+1}(h_1, s_2+2) - d_j(h_2+1,0) & j = 0,
        \end{cases}
    \end{align*}
     which agrees with $a_j(N(E[1])$ if $N(E) = E[1]$ is embanked, but will still make sense if $N(E)$ is only weakly embanked. Note that $N(E)$ will fail to be embanked if and only if $a_0'(N(E[1])) < 0$; in this case, for $T_{E[1]}$ after the second \underline{Halve} rule, the quantity $a_0$ decrements to $-1$ before $n$ can decrement to $0$, at which point a third \underline{Halve} rule has to be applied. From the above expressions, we obtain
    \begin{align*}
        a_j'(N(E[1])) - a_j(E[1]) &= \begin{cases}
            \left(d_{j+1}(h_1,s_2+2) - d_{j+1}(h_1,s_2)\right) + \left(d_j(h_2+1,0) - d_j(h_2,0)\right) & j \geq 1 \\
            \left(d_{j+1}(h_1,s_2+2) - d_{j+1}(h_1,s_2)\right) - \left(d_j(h_2+1,0) - d_j(h_2,0)\right) & j = 0
        \end{cases} \\
        &= \begin{cases}
            d_{j+1}(s_2+2,s_2) + d_j(h_2+1,h_2) & j \geq 1 \\
            d_{j+1}(s_2+2,s_2) - d_j(h_2+1,h_2) & j = 0 
        \end{cases} \\
        &= \begin{cases}
            2d_j(h_2+1,h_2) & j \geq 1 \\
            0 & j = 0
        \end{cases}
    \end{align*}
    where the first equality is justified by the fact that $a_{j+2}(\underline{\text{Zero}}(E[1])) = a_{j+2}(\underline{\text{Zero}}(E))$ by the definition of $E[i]$, and the third equality is justified by verifying that $d_{j+1}(s_2+2,s_2) = d_j(h_2+1,h_2)$ always holds. We learn $a_j'(N(E[1])) - a_j(E[1]) \geq 0$, and therefore that $N(E)$ is embanked. Moreover, we have exactly one index $j \in [0,\ell(E)]$ such that $N(E[1]) = E[1][j]$; this is the index for which $d_j(h_2+1,h_2) > 0$. Such $j$ is characterized by $(h_2+1)/2^j$ being a half-integer, and this occurs precisely when $\nu_2(h_2+1) = j-1$. The desired claim follows for the $i = 1$ case.
    
    If $i=0$, a completely analogous analysis occurs: we compute
    \begin{align*}
        a_j(E[0]) &= \begin{cases}
            a_{j+2}(\underline{\text{Zero}}(E)) + d_{j+2}(2^{\ell}-1,s_1) + d_{j+1}(h_1, s_2) + d_j(h_2,0) & j \geq 1 \\
            a_{j+2}(\underline{\text{Zero}}(E)) + d_{j+2}(2^{\ell}-1,s_1) + d_{j+1}(h_1, s_2) - d_j(h_2,0) & j = 0,
        \end{cases}
    \end{align*}
    and define
    \begin{align*}
        a_j'(N(E[0])) &:= \begin{cases}
            a_{j+2}(\underline{\text{Zero}}(E[0])) + d_{j+2}(2^{\ell}-1,s_1-2) + d_{j+1}(h_1-1, s_2) + d_j(h_2,0) & j \geq 1 \\
            a_{j+2}(\underline{\text{Zero}}(E[0])) + d_{j+2}(2^{\ell}-1,s_1-2) + d_{j+1}(h_1-1, s_2) - d_j(h_2,0) & j = 0,
        \end{cases}
    \end{align*}
    so that for all $j \in [0,\ell]$, we have 
    \begin{align*}
        a_j'(N(E[0])) - a_j(E[0]) &= d_{j+2}(s_1-2,s_1) + d_{j+1}(h_1-1,h_1) \\
        &= 2d_{j+1}(h_1-1, h_1).
    \end{align*}
    As in the $i = 1$ case, the fact that $a_j'(N(E[0])) - a_j(E[0]) \geq 0$ immediately implies that $N(E)$ is embanked, and in fact $a_j'(N(E[0])) - a_j(E[0])$ positive if and only if $h_1/2^{j+1}$ is a half-integer, which occurs precisely when $\nu_2(h_1) = j$. This completes the proof.
\end{proof}
In light of \cref{prop:nextindex}, we are induced to \textit{speed up} the function $N(\cdot)$ to a new function $N'(\cdot)$:
\begin{definition}\label{def:rootedembanked}
    Define the following terms:
    \begin{enumerate}
        \item For an embanked state $E$ such that $N(E) = E[i]$, define $N'(E)$ to be the state
        \begin{align*}
            N'(E) &:= N^{\ceil{(i(E)+1)/2}}(E) \\
            &= \begin{cases}
            E[i][i-2]\cdots[3][1] & i~\text{odd} \\
            E[i][i-2]\cdots[2][0] & i~\text{even},
        \end{cases}
        \end{align*}
        where the second equality holds by \cref{prop:nextindex}.
        \item Say that an embanked state $E$ is \textit{rooted embanked} if it is the result of applying $N'$ to $S_k$ some amount of times, i.e. there exists $e \in \mathbb{N}$ such that $E = (N')^e(S_k)$.

        \item For $i \in \{0,1\}$, say that a rooted embanked state $E$ is \textit{$i$-rooted embanked} if $E = N^{-1}(E)[i]$. Every rooted embanked state is either $0$-rooted embanked or $1$-rooted embanked.

        \item For a state transition $E \to E'$ of rooted embanked states, let $T_{E \to E'}'$ denote the sequence of $N$'s that were applied to $E$ to achieve $E'$.
    \end{enumerate}
\end{definition}
We have the immediate corollary:
\begin{corollary} \label{cor:nextindex}
    For $i \in \{0,1\}$, let $E$ be an $i$-rooted embanked state with $h(N^{-1}(E)) = (h_1,h_2)$, and assume that $N(E)$ is weakly embanked. Then we have 
    \begin{align*}
        N(E) &= \begin{cases}
            E[\nu_2(h_1)] & i = 0 \\
            E[\nu_2(h_2+1)+1] & i = 1,
        \end{cases}
    \end{align*}
    so that $N'(E)$ is $\overline{\nu_2(h_1)}$-rooted embanked if $i = 0$, and $\overline{\nu_2(h_2+1)+1}$-rooted embanked if $i = 1$.
\end{corollary}

\section{Proof of nonhalting}

\begin{proposition} \label{prop:regular}
    Let $E$ be a $1$-rooted embanked state such that there exists $m \in [0,2^{2k}-2)$ of odd $2$-adic valuation such that $h(N^{-1}(E)) = (2^{2k}-m-1, 2^{2k}+m)$. Then the following hold:
    \begin{enumerate}
        \item If $m' > m$ in $[0,2^{2k}-2]$ denotes the next number after $m$ satisfying $\nu_2(m') \equiv 1 \bmod 2$, then there exists a unique 1-rooted embanked state $E'$ satisfying  $E \mapsto E'$, such that $h(N^{-1}(E')) = (2^{2k}-m'-1, 2^{2k}+m')$.
        \item In the setting of (2), let $d := m'-m$. Then $a_i(E') = a_i(E) + 2d$ for $i \in \{0,1\}$.
    \end{enumerate}
\end{proposition}
\begin{proof}
     Suppose $E$ and $m$ are as in the problem statement. In what follows, the quantities $a_1$ and $a_0$ never breach the bounds given by the conditions in \cref{prop:weaklyembankedbounds}; combining this with \cref{cor:nextindex}, it follows that all $N'$-iterates of $E$ we consider will be rooted embanked. 
     
     We claim that for each $e \in [0,m'-m-1]$, $E_e := (N')^e(E)$ is $1$-rooted embanked with $h(N^{-1}(E_e)) = (2^{2k}-m-1, 2^{2k}+m+e)$. This follows by induction on $e$: the base case $e = 0$ is given, while for the induction step, if $e < m'-m-1$ satisfies the induction hypothesis, then since $m+e+1 < m'$, it must necessarily have even valuation; hence by \cref{cor:nextindex}, $E_{e+1} := N'(E_e)$ is $\overline{\nu_2((m+e)+1)+1} = 1$-rooted embanked, and moreover by \cref{lemma:change} we have $h(N^{-1}(E_{e+1})) = h(N^{-1}(E_e)[1]) = (2^{2k}-m-1, 2^{2k}+m+e+1)$. This completes the induction. In addition, we also learn that from $E$ to $E_{m'-m-1}$, the quantity $a_1$ increments by $2(m'-m-1).$

    When we apply $N'$ to $E_{m'-m-1}$, we obtain a state $E'$ such that $h(N^{-1}(E')) = (2^{2k}-m-1, 2^{2k}+m')$. Since $\nu_2((m'-1)+1) = \nu_2(m')$ is odd, it follows by \cref{cor:nextindex} that $E' = N'(E_{m'-m-1})$ is $\overline{\nu_2((m'-1)+1)+1} = 0$-rooted. We claim that for each $e \in [0,m'-m-1]$, $E_e' := (N')^e(E')$ is $0$-rooted embanked with $h(N^{-1}(E_e')) = (2^{2k}-m-1-e, 2^{2k}+m')$. This, again, follows by induction on $e$: the base case $e = 0$ is given, while for the induction step, if $e < m'-m-1$ satisfies the induction hypothesis, then since $m+e+1 < m'$, it must necessarily have even valuation; hence by \cref{cor:nextindex}, $E_{e+1}' := N'(E_e')$ is $\overline{\nu_2(2^{2k}-m-1-e)} = 0$-rooted embanked, and moreover by \cref{lemma:change} we have $h(N^{-1}(E_{e+1}')) = h(N^{-1}(E_e')[0]) = (2^{2k}-m-2-e, 2^{2k}+m')$. This completes the induction. In addition, we also learn that from $E_{m'-m-1}$ to $E_{m'-m-1}'$, the quantity $a_0$ decrements by $2(m'-m).$

    Finally, applying $N'$ to $E_{m'-m-1}'$ yields a state $E''$ such that $h(N^{-1}(E'')) = (2^{2k}-m'-1, 2^{2k}+m')$, and $E''$ is $\overline{\nu_2(2^{2k}-m')} = 1$-rooted embanked, which proves (1). This extra application of $N'$ also increments $a_1$ by $2$, so in total, $a_1$ has incremented by $2(m'-m-1)+2 = 2d$, while $a_0$ has decremented by $2(m'-m) = 2d$; this proves (2).
\end{proof}
Starting with $S_k = (0,2,2^2,\dots,2^{2k},0)$, a brief computation yields $h(S_k) = (2^{2k}-1,2^{2k})$ and $N(S_k) = S_k[2k+1]$. Thus, applying $N'$ to $S_k$ five times yields a $1$-rooted embanked state $E := (N')^5(S_k)$ such that $h(N^{-1}(E)) = (2^{2k}-3, 2^{2k}+2)$. Hence, by inducting on $m$ to reach $m = 2^{2k}-2$, \cref{prop:regular} and \cref{prop:weaklyembankedbounds} imply the following corollary:
\begin{corollary}
    There exists a rooted embanked state $E$ such that $h(N^{-1}(E)) = (1, 2^{2k+1}-2)$.
\end{corollary}

\subsection{Counting increments}
For rooted embanked states $E \mapsto E'$, let $\kappa_{E \to E'}(j)$ denote the number of $N\colon S \mapsto N(S)$ in $T_{E \to E'}'$ such that $N(S) = S[j]$.

\begin{proposition}
    Suppose we are in the setting of \cref{prop:regular}, with $E$, $E'$, $m$ and $m'$ defined as before. Then, for all $j \in [0,2k+1]$, we have $\kappa_{E \to E'}(j) = \#\{e \in [m+1,m'] \colon \max(1,2^{j-1})\mid e\}$.
\end{proposition}
\begin{proof}
    For a statement $X$, define $\delta(X)$ to be $1$ if $X$ is true, and $0$ if $X$ is false. Fix a $j \in [0,2k+1]$. By \cref{cor:nextindex}, we have 
    \begin{align*}
        \kappa_{E \to E'}(j) &= \begin{cases}
          \delta(\nu_2((m'-1)+1)+1 \geq j) + \#\{e \in [m+1, m'-1]\colon \nu_2(e) \geq j\}   & j \equiv 0 \bmod 2 \\
         \#\{e \in [m, m'-2]\colon \nu_2(e+1)+1 \geq j\} + \delta(\nu_2(m') \geq j)   & j \equiv 1 \bmod 2
        \end{cases} \\
        &= \begin{cases}
          \delta(2^{j-1}\mid m') + \#\{e \in [m+1, m'-1]\colon 2^j \mid e\}   & j \equiv 0 \bmod 2 \\
          \#\{e \in [m+1, m'-1]\colon 2^{j-1} \mid e\} + \delta(2^j \mid m')   & j \equiv 1 \bmod 2
        \end{cases} \\
        &= \begin{cases}
          \delta(2^{j-1}\mid m') + \#\{e \in [m+1, m'-1]\colon 2^{j-1} \mid e\}   & j \equiv 0 \bmod 2 \\
          \#\{e \in [m+1, m'-1]\colon 2^{j-1} \mid e\} + \delta(2^{j-1} \mid m')   & j \equiv 1 \bmod 2
        \end{cases} \\
        &= \#\{e \in [m+1, m']\colon 2^{j-1} \mid e\},
    \end{align*}
    where the third equality is using the fact that $m$ and $m'$ have odd $2$-adic valuations, but every number in between has even $2$-adic valuation.
\end{proof}

\subsection{Endgame}
Our setup for the end of the proof is the sequence of state transitions $S_k \to E \to E' \to E''$, where $E$, $E'$ and $E''$ are all rooted embanked states satisfying $h(N^{-1}(E)) = (2^{2k}-3, 2^{2k}+2)$, $h(N^{-1}(E')) = (1, 2^{2k+1}-2)$ and $h(N^{-1}(E'')) = (1, 2^{2k+1}-1)$. In particular, we have $E' \to E''$ since \cref{prop:regular}(2) guarantees that $E'$ will satisfy the bounds in \cref{prop:weaklyembankedbounds} so that $E'$ is weakly embanked, and then \cref{cor:nextindex} implies that $E'$ is embanked. Also, we have $E = (N')^5(S_k)$. An explicit computation yields 
\begin{align*}
    \kappa_{S_k \to E}(j) &= \begin{cases}
        2 & j = 0 \\
        3 & j = 1 \\
        1 & j = 2 \\
        \bar{j} & 3 \leq j \leq 2k+1,
    \end{cases} \\
    \kappa_{E \to E'}(j) &= \#\{e \in [3,2^{2k}-2]\colon 2^{j-1}\mid e\} \\
    &= \begin{cases}
        2^{2k}-4 & j = 0,1 \\
        2^{2k-1}-2 & j = 2 \\
        2^{2k-j+1}-1 & j \in [3,2k] \\
        0 & j = 2k+1,
    \end{cases} \\
    \kappa_{E' \to E''}(j) &= \begin{cases}
        1 & j = 1 \\
        0 & j \neq 1.
    \end{cases}
\end{align*}
It follows that 
\begin{align*}
    \kappa_{S_k \to E''}(j) &= \kappa_{S_k \to E}(j) + \kappa_{E \to E'}(j) + \kappa_{E' \to E''}(j) \\
    &= \begin{cases}
        2^{2k}-2 & j = 0 \\
        2^{2k} & j = 1 \\
        2^{2k-j+1}-1 + \overline{j} & j \in [2,2k] \\
        1 & j = 2k+1.
    \end{cases}
\end{align*}
Since moving from $S$ to $S[j]$ increments $a_j$ by $2$, we obtain
\begin{align*}
    a_j(E'') &= a_j(S_k) + 2\kappa_{S_k \to E''}(j) \\
    &= \begin{cases}
        2^{2k+1} - 4 & j = 0 \\
        3\cdot 2^{2k} & j = 1 \\
        3 \cdot 2^{2k-j+1} - 2 + 2\overline{j} & j \in [2,2k] \\
        2 & j = 2k+1.
    \end{cases}
\end{align*}
Observe that $E''$ fails the conditions of \cref{prop:weaklyembankedbounds}, and so it will not make sense to apply $N$ anymore. We resort to an explicit analysis from here. From the state $E''$, applying \underline{Zero} yields 
\begin{align*}
    (0,0,3, 3\cdot 2^1-2, 3\cdot 2^2, 3\cdot 2^3-2, \dots, 3\cdot 2^{2k-1}-2, 3\cdot 2^{2k}, 2^{2k+1}-5) &(n,\sigma) = (2^{2k+1}-1,-1)
\end{align*}
after which applying $2^{2k+1}-4$ \underline{Increment} rules yields 
\begin{align*}
    &(0,0,2^2,2^3-2,2^4,2^5-2,\dots,2^{2k-1}-2, 2^{2k},2^{2k+1}-3,2^{2k+2}-2, -1)  &(n,\sigma)=(3,-1),
\end{align*}
and now we \underline{Halve} to obtain
\begin{align*}
    &(0,0,2^2,2^3-2,2^4,2^5-2,\dots,2^{2k-1}-2, 2^{2k},2^{2k+1}-3,2^{2k+2}-2). &(n,\sigma) = (1,+1).
\end{align*}
Applying $2^{2k+2}-2$ \underline{Increment} rules yields
\begin{align*}
    &(1,2,2^3,2^4-2,2^5,2^6-2,\dots,2^{2k}-2,2^{2k+1},2^{2k+2}-4,0) &(n,\sigma) = (2^{2k+2}-1, +1)
\end{align*}
after which we are forced to apply an \underline{Overflow} rule to get 
\begin{align*}
    &E_{k,\text{final}} :=(0,2,2,2^3,2^4-2,\dots,2^{2k}-2,2^{2k+1},2^{2k+2}-4,0) &(n,\sigma)=(0,-1).
\end{align*}
By \cref{prop:weaklyembankedbounds}, $E_{k,\text{final}}$ is weakly embanked, and one may compute that $h(E_{k,\text{final}}) = (2^{2k+2}-1,2^{2k+2}-2)$. Further computing the state formed after the second \underline{Halve} rule of $T_{E_{k,\text{final}}}$ shows that $E_{k,\text{final}}$ is in fact embanked with $N(E_{k,\text{final}})= 2k+1$. By \label{def:rootedembanked}(1), applying $N'(\cdot)$ to $E_{k,\text{final}}$ yields 
\begin{align*}
    E_{k,\text{final}}' &= (0,2,2^2,2^3,2^4,\dots,2^{2k},2^{2k+1},2^{2k+2}-2,0).
\end{align*}
By \cref{lemma:change}, we have $h(E_{k,\text{final}}') = (2^{2k+2}-1,2^{2k+2}-1)$; therefore, \cref{prop:nextindex} tells us that $N(E_{k,\text{final}}') = E_{k,\text{final}}'[\nu_2(2^{2k+2}-1)+1] = E_{k,\text{final}}'[1]$. Hence,
\begin{align*}
    N(E_{k,\text{final}}') &= (0,2,2^2,2^3,2^4,\dots,2^{2k},2^{2k+1},2^{2k+2},0),
\end{align*}
which is just $S_{k+1}$.

\printbibliography

\end{document}